\documentclass[10pt,reqno]{amsart}

\usepackage[utf8]{inputenc}
\usepackage[english]{babel}
\usepackage{amsmath,amsfonts,amssymb,amsthm,shuffle}
\usepackage[T1]{fontenc}
\usepackage{lmodern}
\usepackage{mathtools}
\usepackage[titletoc]{appendix}

\usepackage[top=3.5cm,bottom=3.5cm,left=3.6cm,right=3.6cm]{geometry}

\usepackage[dvipsnames]{xcolor}
\usepackage[hyperindex=true,frenchlinks=true,colorlinks=true,
citecolor=Mahogany,linkcolor=DarkOrchid,urlcolor=Tan,linktocpage,
pagebackref=true]{hyperref}

\usepackage{tikz}
\usetikzlibrary{shapes}
\usetikzlibrary{fit}
\usetikzlibrary{decorations.pathmorphing}
\usepackage[textsize=footnotesize]{todonotes}
\usepackage{dsfont}
\usepackage{wasysym}
\usepackage{stmaryrd}
\usepackage{cite}
\usepackage{subfigure}
\usepackage{multirow}
\usepackage{enumitem}
\usepackage{multicol}
\usepackage{bm}

\title[]{An exponential inequality for orthomartingale difference random fields 
and some applications}
\date{\today}
\author{Davide Giraudo}
\thanks{ Acknowledgement: “The work is supported by Ministry of Science and 
Higher Education of the Russian Federation, agreement N 075–15–2019–1619“}
\address{Saint-Petersburg University}

\numberwithin{equation}{section}
\renewcommand{\leq}{\leqslant}
\renewcommand{\geq}{\geqslant}

\newtheorem{Theorem}{Theorem}[section]
\newtheorem{Th\'eor\`eme}{Th\'eor\`eme}[section]
\newtheorem{Proposition}[Theorem]{Proposition}
\newtheorem{Lemma}[Theorem]{Lemma}

\newtheorem{Definition}[Theorem]{Definition}
\newtheorem{D\'efinition}[Th\'eor\`eme]{D\'efinition}

\theoremstyle{remark}
\newtheorem{Remark}[Theorem]{Remark}
\newtheorem{Example}[Theorem]{Example}

\tikzstyle{Vertex}=[circle,draw=LimeGreen!80,fill=LimeGreen!8,
inner sep=1pt,minimum size=2mm,line width=1pt,font=\scriptsize]
\tikzstyle{Node}=[Vertex,draw=RoyalBlue!80,fill=RoyalBlue!8,inner sep=1.5pt]
\tikzstyle{Leaf}=[rectangle,draw=Black!70,fill=Black!16,
inner sep=0pt,minimum size=1mm,line width=1.25pt]
\tikzstyle{Edge}=[Maroon!80,cap=round,line width=1pt]
\tikzstyle{Mark1}=[draw=BrickRed!80,fill=BrickRed!8]
\tikzstyle{Mark2}=[draw=BurntOrange!80,fill=BurntOrange!8]
\tikzstyle{EdgeRew}=[->,RedOrange!80,cap=round,thick]

\newcommand{\Fca}{\mathcal{F}}
\newcommand{\Gca}{\mathcal{G}}
\newcommand{\Hca}{\mathcal{H}}

\newcommand{\Rca}{\mathcal{R}}

\newcommand \ens[1]{\left\{ #1\right\}}
\newcommand \R{\mathbb R}

\newcommand \N{\mathbb N}
\newcommand \PP{\mathbb P}

\newcommand{\E}[1]{\mathbb E\left[#1\right]}

\newcommand \Z{\mathbb Z}

\newcommand \abs[1]{\left|#1\right|}
\newcommand \eps{\varepsilon}

\newcommand{\f}{\mathcal F}

\newcommand{\pr}[1]{\left(#1\right)}
\newcommand{\norm}[1]{\left\lVert #1 \right\rVert}
\newcommand{\gr}[1]{\bm{#1}}
\newcommand{\gri}{\bm{i}}
\newcommand{\grj}{\bm{j}}
\newcommand{\grn}{\bm{n}}
\newcommand{\grm}{\bm{m}}
\newcommand{\gru}{\bm{u}}

\newcommand{\grk}{\bm{k}}
\newcommand{\grt}{\bm{t}}
\newcommand{\imd}{\preccurlyeq}
\newcommand{\smd}{\succcurlyeq }

\begin{document}
\maketitle 
\begin{abstract}
 In this paper, we establish an exponential inequality for random fields, which 
is applied in the context of convergence rates in the law of large numbers and 
Hölderian weak invariance principle.
\end{abstract}
 
\section{Statement of the exponential inequality}

\subsection{Goal and motivations}
Understanding the asymptotic behavior of sums of random variables 
is an important topic in probability. A useful tool for 
establishing limit 
theorems is to control the probability that the partial sums or the 
maximum of the absolute values of the partial sums exceeds a fixed 
number. A direct 
application can give convergence rates in the law of large 
numbers, and such an inequality 
can be used to check tightness criteria in some functional spaces. 

In this paper, we will be concerned by establishing
and applying an exponential inequality 
for the so-called orthomartingale difference random fields 
introduced in \cite{MR0254912}. 
We make an assumption of invariance in law of the partial sums, 
which is weaker than stationarity. 
We control the tails of maximum of partial sums on rectangles 
of $\N^d$ by two quantities:
an exponential term and another one involving the common
distribution function of the increments. We will then formulate two 
applications of the result. One will deal with the convergence rates in the 
strong law of 
large numbers. The second one will be about the invariance 
principle in Hölder spaces, which will be restricted to the case of stationary 
processes for technical reasons. 

Orthomartingales are well-suited for the 
summation over rectangles of $\Z^d$ because 
we can use a one dimensional martingale property when 
we sum according to each fixed coordinate. Moreover, 
the approximation of stationary random fields by 
this class of martingales can be done, like 
in \cite{MR3798239,MR4166203,MR3504508,MR3869881}. 
However, for rates in the law of large numbers and the functional central limit 
theorem in Hölder spaces, only a few results 
are available in the literature. Indeed, rates on the law 
of large number for orthomartingales with polynomial moment 
have been given in \cite{MR2794415,MR4046858,MR3451971}, but 
it seems that the question of exponential moments 
was only addressed for martingale difference sequences. 
For the functional central limit theorem in Hölder spaces 
for random fields, the i.i.d.\ case was addressed 
for moduli of regularity of the form $t^\alpha$, $0<\alpha<1/2$ 
in the case of random fields, and for more general moduli 
in the case of sequences. It turns out that the inequality 
we will present in this paper is a well-suited tool for dealing 
with these problems.

A key ingredient for the aforementioned limit theorems is 
an exponential inequality for orthomartingale random fields, 
that is, an inequality putting into play an exponential term 
and the tail of the common distribution of the increments. 
At first glance, it seems that multi-indexed martingales 
could be treated like sequences, up to some technical and 
notational obstacles. However, it turns out that the standard 
tools for proving deviation inequalities for martingales, like 
martingale transforms or uses of exponential supermartingale, 
do not extend easily. Nevertheless, it is possible to apply induction arguments 
when the random field satisfies good 
properties. 

The paper is organized as follows: in Subsection~\ref{subsec:expo_ineg}, we 
state the definition of orthomartingale difference random fields and an 
exponential inequality for such 
random fields, with an assumption of invariance of the law of 
the sum on rectangles by translation. Subsections~\ref{subsec:LLN} 
(respectively \ref{subsc:HWIP}) provide an application to the 
rates in the strong law of large numbers (respectively, to 
the invariance principle in Hölder spaces). Section~\ref{sec:proofs} is devoted 
to the proof of the previously stated 
results.

\subsection{Exponential inequality for orthomartingales}
\label{subsec:expo_ineg}

We start by defining the concept of orthomartingale. Before that, 
we need to introduce the notion of filtration with respect to the 
coordinatewise order $\imd$: we say that for $\gri,\grj\in\Z^d$, 
$\gri\imd\grj$ if $i_q\leq j_q$ for each $q\in \ens{1,\dots,
d}=:[d]$.

\begin{Definition}
A filtration indexed by $\Z^d$ on a probability 
space $\pr{\Omega,\f,\PP}$ is a collection of 
$\sigma$-algebras $\pr{\f_{\gri}}_{\gri\in \Z}$ such that 
for each $\gri\imd\grj$, the inclusion $\f_{\gri}\subset 
\f_{\grj}$ takes place.
\end{Definition}
\begin{Definition}
We say that the filtration
$\pr{\f_{\gri}}_{\gri\in \Z}$ is commuting if for all 
integrable random variable $Y$ and all $\gri$, $\grj\in\Z$, the equality 
\begin{equation}
\E{\E{Y\mid \f_{\gri}}\mid\f_{\grj}}
= \E{Y\mid \f_{\min\ens{\gri,\grj}}},
\end{equation}
takes place, where the minimum is taken coordinatewise.
\end{Definition}

\begin{Example}
Let $\pr{\eps_{\grj}}_{\grj\in\Z^d}$ be an i.i.d.\ random 
field and let $\f_{\grj}:=\sigma\pr{\eps_{\gri},\gri\imd\grj}$. 
Then the filtration  $\pr{\f_{\gri}}_{\gri\in \Z}$ is commuting.
\end{Example}
\begin{Example}
Let $\pr{\eps_k^{\pr{q}}}_{k\in\Z}$, $q\in [d]:=\ens{1,\dots,d}$ be independent 
copies of an i.i.d.\ sequence $\pr{\eps_k}_{k\in\Z}$ and let 
$\f_{\grj}:=\sigma\pr{\eps_{k_q}^{\pr{q}},k_q\leq j_q,1\leq q\leq d}$. Then the 
filtration  $\pr{\f_{\gri}}_{\gri\in \Z}$ is commuting, and connected to 
decoupled $U$-statistics.
\end{Example}

\begin{Definition}
 We say that the random field $\pr{X_{\gri}}_{\gri\in \Z^d}$ 
 is an orthomartingale difference random field 
 with respect to the commuting 
 filtration $\pr{\f_{\gri}}_{\gri\in\Z^d}$ if for all 
 $\gri\in\Z^d$, $X_{\gri}$ is integrable, 
 $\f_{\gri}$-measurable and 
 \begin{equation}
  \E{X_{\gri}\mid \f_{\gri-\gr{e}_q}}=0,
 \end{equation}
where for $q\in [d]$,  $\gr{e}_q$ is the $q$-th vector of the canonical 
basis of $\R^d$.
\end{Definition}

\begin{Theorem}\label{thm:deviation_inequality_orthomartingale}
 Let $\pr{X_{\gri}}_{\gri\in\Z^d}$ be an orthomartingale 
 differences random field 
 such that for all $\grn\in\N^d$ and all $\grk\in\Z^d$, 
 \begin{equation}\label{eq:invariance_loi_sommes}
  S_{\grn}:=\sum_{\gr{1}\imd\gri\imd\grn}
  X_{\gri}\mbox{ and } \sum_{\gr{1}\imd\gri\imd\grn}
  X_{\gri+\grk}\mbox{ have the same distribution}.
 \end{equation}
Then the following inequality holds for all $x,y>0$:
\begin{multline}\label{eq:exp_deviation_inequality_orthomartingale}
 \PP\ens{\max_{\gr{1}\imd\gri\imd\gr{n}  } \abs{S_{\gri}  }
 >x\abs{\grn}^{1/2}}\leq A_d\exp\pr{-\pr{\frac{x}{y}}^{2/d}}
 \\+B_d\int_1^{+\infty}\PP\ens{\abs{X_{\gr{1}}}>yuC_d}
 u\pr{\log\pr{1+u}}^{p_d}du,
\end{multline}
where $A_d$, $B_d$ and $C_d$ depend only on $d$,  
$p_d=2d$ and $\abs{\grn}=\prod_{q=1}^dn_q$.
\end{Theorem}

\begin{Remark}
Assume that $\pr{X_{\gri}}_{\gri\in\Z^d}$ is bounded, that is, there 
exists a constant $K$ such that $\abs{X_{\gr{i}}}\leq K$ 
almost surely for all $\gri\in\Z^d$. If $x>3^{d/2}K/C_d$ then 
we can choose $y=K/C_d$   and inequality 
\eqref{eq:exp_deviation_inequality_orthomartingale} simplifies 
as 
\begin{equation}
  \PP\ens{\max_{\gr{1}\imd\gri\imd\gr{n}  } \abs{S_{\gri}  }
 >x\abs{\grn}^{1/2}}\leq A_d\exp\pr{-\pr{\frac{C_dx}{K}}^{2/d}}.
\end{equation}

\end{Remark}

\begin{Remark}
The exponent $2/d$ in the exponential term 
of \eqref{eq:exp_deviation_inequality_orthomartingale}
is not improvable, even in the bounded case. To see this, 
consider $d$ i.i.d.\ sequences $\pr{\eps_i^{\pr{q}}}_{i\in\Z}$ 
which are independent of each other, in the sense  
that the collection $\ens{\eps_{i_q}^{\pr{q}}, 
i_q\in\Z}$ is independent. Assume that $\eps_0^{\pr{q}}$ 
takes the values $1$ and $-1$ with probability $1/2$ for all 
$q\in [d]$. Let $X_{\gri}=\prod_{q=1}^d\eps_{i_q}^{\pr{q}}$. 
Then 
\begin{equation}
  \PP\ens{\max_{\gr{1}\imd\gri\imd\gr{n}  } \abs{S_{\gri}  }
 >x\abs{\grn}^{1/2}}\geq
  \PP\ens{ \abs{S_{\grn}  }
 >x\abs{\grn}^{1/2}}=
 \PP\ens{ \abs{\prod_{q=1}^d
 \frac 1{\sqrt{n_q}}\sum_{i_q=1}^{n_q}\eps_{i_q}^{\pr{q}} 
 }
 >x }.
\end{equation}
The vector $\pr{\frac 
1{\sqrt{n_q}}\sum_{i_q=1}^{n_q}\eps_{i_q}^{\pr{q}}}_{q=1}^d$ converges in 
distribution to 
$\pr{N_q}_{q=1}^d$, where $\pr{N_q}_{q=1}^d$ is independent 
and each $N_q$ has a standard normal distribution. Therefore, 
if $f\pr{x}$ is a function such that for each bounded by $1$
orthomartingale difference random field 
satisfying \eqref{eq:invariance_loi_sommes} and each $x>0$
\begin{equation}
 \PP\ens{\max_{\gr{1}\imd\gri\imd\gr{n}  } \abs{S_{\gri}  }
 >x\abs{\grn}^{1/2}}\leq f\pr{x},
\end{equation}
then letting $Y:= \prod_{q=1}^d\abs{N_q}$,
the following inequality should hold for all $x$:
$\PP\ens{Y>x}\leq f\pr{x}$. Since the $\mathbb L^p$-norm
of $Y$ behaves like $p^{d/2}$, the function $f$ cannot decay 
quicker than $\exp\pr{-Kx^{\gamma}}$ for some $K>0$ and $\gamma>2/d$.
\end{Remark}

\begin{Remark}
 The $\log$ factor in the right hand side of 
\eqref{eq:exp_deviation_inequality_orthomartingale} appears 
naturally as iterations of weak-type estimates of the form 
$x\PP\ens{X>x}\leq \E{X\mathbf{1}\ens{Y>x}}$ for some random variable $Y$, 
giving a control of the tail of $X$ in terms of that of $Y$. Consequently, the 
$\log$
factor does not seem to be avoidable with this method of proof. 
We do not know whether this factor can be removed.
\end{Remark}

\section{Application to limit theorems}

\subsection{Convergence rates in the law of large numbers}
\label{subsec:LLN}
A centered sequence $\pr{X_i}_{i\geq 1}$ satisfies the law of large 
numbers if the sequence $\pr{n^{-1}\sum_{i=1}^nX_i}_{i\geq 1}$ 
converges almost surely to zero. This is for example the 
case of a strictly stationary ergodic sequence where 
$X_1$ is integrable and centered. Then arises the question of 
evaluating the speed of convergence and finding 
bounds for the large deviation probabilities, 
namely, 
\begin{equation}
 \PP\ens{\frac 1n\abs{ \sum_{i=1}^nX_i}
 >x}.
\end{equation}
This question has been treated in the independent case 
in \cite{MR19852,MR79851,MR30714} under conditions 
on the $\mathbb L^p$-norm of $X_i$. The case of martingale 
differences has also been addressed: under boundedness of 
moments of order $p$ and a Cramer-type condition 
$\sup_{i\geq 1}\E{\exp\pr{\abs{X_i}}}<+\infty$ in 
\cite{MR1856684}, a conditional Cramer condition in 
\cite{MR2568267} and under finite exponential 
moments in \cite{MR3005732}.

For random fields, we can consider large 
deviation probabilities defined by 
\begin{equation}
 \PP\ens{\frac 1{\abs{\gr{N}}}\abs{  \sum_{\gr{1}\imd 
 \grn\imd\gr{N}} X_{\gri}}
 >x}, \gr{N} \smd\gr{1}, x>0.
\end{equation}
Result for orthomartingales with polynomial moment 
have been given in \cite{MR2794415,MR4046858,MR3451971}.

\begin{Theorem}\label{thm:large_deviation}
Let $\pr{X_{\gri}}_{\gri\in\Z^d}$ be an orthomartingale 
difference random field satisfying \eqref{eq:invariance_loi_sommes}. Suppose 
that for some 
$\gamma>0$, 
\begin{equation}
 \sup_{s>0}\exp\pr{s^\gamma}\PP\ens{\abs{X_{\gr{1}}}>s}\leq 2.
\end{equation}
Then for each positive $x$, the following inequality takes 
place:
\begin{equation}
  \PP\ens{\frac 1{\abs{\gr{N}}}\abs{  \sum_{\gr{1}\imd 
 \gr{n}\imd\gr{N}} X_{\gri}}
 >x}\leq C_{1,d,\gamma}\exp\pr{-C_{2,d,\gamma}
 \abs{\gr{N}}^{\frac{\gamma}{2+d\gamma}}x^{\frac{2\gamma}{2+d\gamma}}  },
\end{equation}
where $C_{1,d,\gamma}$ and $C_{2,d,\gamma}$ depend only on $d$
 and $\gamma$.
\end{Theorem}

\subsection{Application to Hölderian invariance principle}
\label{subsc:HWIP}
Given a sequence of random variables $\pr{X_i}_{i\geq 1}$, 
a way to understand the asymptotic behavior of 
the partial sums given by $S_k=\sum_{i=1}^kX_i$ is to 
define a sequence of random functions $\pr{W_n}_{n\geq 1}$ 
on the unit interval $[0,1]$ in the following way: $W_n\pr{k/n}
=n^{-1/2}S_k$ for $k\in\ens{1,\dots,n}$, $W_n\pr{0}=0$ and on 
$\pr{\pr{k-1}/n,k/n}$, $W_n$ is linearly interpolated. 
When the sequence $\pr{X_i}_{i\geq 1}$ is i.i.d., centered 
and has unit variance, the sequence $\pr{W_n}_{n\geq 1}$ 
converges in law in $C[0,1]$ to a standard Brownian motion. The 
result has been extended 
to strictly stationary martingale difference sequences 
in \cite{MR0126871,MR0151997}. Then numerous papers 
treated the case of weakly dependent strictly stationary 
sequence, see \cite{MR2206313} and the references 
therein for an overview. 

There are also two other possibilities of extension of 
such result. The first one is to consider other 
functional spaces, in order to establish the convergence 
of $\pr{F\pr{W_n}}_{n\geq 1}$ for a larger class 
of functionals than the continuous functionals  
$F\colon C[0,1]\to\R$. In other words, we view 
$W_n$ as an element of an element of a Hölder 
space $\mathcal H_\rho$ with modulus of regularity 
$\rho$ and investigate the convergence of 
$\pr{W_n}_{n\geq 1}$ in this function space. 
Applications to epidemic changes can be given, for 
example in \cite{MR2265057}. The question of 
the invariance principle in Hölder spaces has 
been treated for i.i.d.\ sequences for modulus of 
regularity of the form $t^\alpha$ in \cite{MR2000642}
and also for more general ones, of the form $t^{1/2}\log
\pr{ct}^\beta$ in \cite{MR2054586}. Some results are also 
available for stationary weakly dependent sequences, 
for example mixing sequences \cite{MR1759810,MR3615086} 
or by the use of a martingale approximation 
\cite{MR3426520,MR3770868}.

A second development of these invariance principles 
is the consideration of partial sum processes built on 
random fields. Given a random field $\pr{X_{\gri}}_{\gri\in 
\Z^d}$, one can define the following random function 
on $[0,1]^d$:
\begin{equation}
 W_{\grn}\pr{\grt}=\frac 1{\sqrt{\abs{\grn}}}
 \sum_{\gri \in \Z^d}
 \lambda\pr{R_{\gri}\cap \prod_{q=1}^d
 [0,n_qt_q]}
 X_{\gri}, \grn\smd\gr{1}, \grt \in [0,1]^d,
\end{equation}
where $\lambda$ denotes the Lebesgue measure on 
$\R^d$ and $R_{\gri}=\prod_{q=1}^d\pr{
i_q-1,i_q}$. Notice that the map $\grt\mapsto 
W_{\grn}\pr{\grt}$ is $1$-Lipschitz continuous and that 
$W_{\grn}\pr{\pr{\frac{k_q}{n_q}}_{q=1}^d}=S_{k_1,\dots,k_d}/\sqrt{\abs{\grn}}$ 
hence $W_{\grn}$ takes into account the values of all the partial 
sums $S_{\grk}$ for $\gr{1}\imd\gr{k}\imd\grn$.

Convergence of $W_{\grn}$ in the space of continuous functions 
has been investigated in \cite{MR0246359} for an i.i.d.\ random field, 
\cite{MR1875665} for martingales for the lexicographic order 
and in \cite{MR3427925,MR3913270} for orthomartingales. 

In this section, we will study the convergence of 
$W_{\grn}$ in some Hölder spaces for orthomartingales 
random fields. First, an observation is that for 
$\grt,\gr{t'}\in [0,1]^d$, 
\begin{equation}
\abs{ \lambda\pr{R_{\gri}\cap \prod_{q=1}^d
 [0,n_qt_q]}- \lambda\pr{R_{\gri}\cap \prod_{q=1}^d
 [0,n_qt'_q]}}\leq \max_{1\leq q\leq d}n_q 
 \abs{t_q-t'q}\leq \abs{\grn}\norm{\grt-\gr{t'}}_\infty
\end{equation}
hence 
\begin{equation}
 \abs{W_{\grn}\pr{\grt}-W_{\grn}\pr{\gr{t'}}}
 \leq \sqrt{\abs{\grn}} \norm{\grt-\gr{t'}}_\infty
 \sum_{\gr{1}\imd \gri \imd\grn}\abs{X_i}
\end{equation}
hence for almost every $\omega$, the map 
$\grt\mapsto W_{\grn}\pr{\grt}\pr{\omega}$ is 
Lipschitz-continuous. As pointed out in the i.i.d.\ case in  
\cite{MR0246359}, the finite dimensional distributions  
of $W_{\grn}$ converge to those of a standard Brownian 
sheet, that is, a centered Gaussian process $\pr{W\pr{\grt}}_{\grt 
\in [0,1]^d}$ whose covariance function is given by 
\begin{equation}
 \operatorname{Cov}\pr{W\pr{\grt},W\pr{\gr{t'}}}
 =\prod_{q=1}^d\min\ens{t_q,t'_q},
\end{equation}
provided that $X_{\gr{1}}$ is centered and has unit variance.
Given an increasing continuous function $\rho\colon [0,1]\to\R$ 
such that $\rho\pr{0}=0$, 
let $\Hca_\rho$ be the space of all the functions $x$ from 
$[0,1]^d$ to $\R$ such that $\sup_{\grt,\gr{t'}\in [0,1]^d,
\grt\neq\gr{t'}}\abs{x\pr{\grt}-x\pr{\gr{t'}}}/
\rho\pr{\norm{\grt-\gr{t'}}_\infty }$ is finite.

The Brownian sheet has trajectories in $\Hca_\rho$ with $\rho\pr{h}
=h^\alpha$ for all $\alpha \in\pr{0,1/2}$ but not for $\alpha\geq 1/2$. 
Therefore, it is not possible to expect to show the convergence of 
$\pr{W_{\grn}}_{\grn\smd\gr{1}}$ in all the possible Hölder spaces 
and some restriction have to be made. In order to treat a 
class of modulus of regularity larger than the power function, 
we need to introduce the slowly varying functions. We say that 
a function $L\colon \pr{0,+\infty}\to \pr{0,+\infty}$ is slowly 
varying if for all $c>0$, the quantity $L\pr{ct}/L\pr{t}$ goes 
to $1$ as $t$ goes to infinity. For example, functions which 
behave asymptotically as a power of the logarithm are slowly varying.
We now define the moduli of regularity and the 
associated Hölder spaces of interest from the point of view of 
the convergence of the partial sum process $W_{\grn}$.

\begin{Definition}
 Let $d\geq 1$ be an integer. We say that $\rho$ 
 belongs to the class $\Rca_{1/2,d}$ if there exists 
 a slowly varying function $L$ such that 
 $L\pr{t}\to \infty$ as $t$ goes to infinity
 and a constant $c$ such that 
 \begin{equation}\label{eq:def_de_rho}
  \rho\pr{h}=h^{1/2}\pr{\ln\pr{\frac{c}{h}}}^{d/2}L\pr{\frac 1h}, h 
  \in [0,1],
 \end{equation}
and $\rho $ is increasing on $[0,1]$.
\end{Definition}
It seems that the exponent $d/2$ is the best we can get in view of 
proving tightness via the deviation inequality we established. It may not be 
optimal in some cases, for example, if $\pr{X_{\gri}}_{\gri\in\Z^d}$ is 
centered, we can get a similar inequality as 
\eqref{eq:exp_deviation_inequality_orthomartingale}, but with the 
exponent $2$ instead of $2/d$ in the right hand side. As a consequence, 
tightness in $\Hca_{\rho}$ can  be established for $\rho\in\Rca_{1/2,1}$, a  
larger class than $\rho\in\Rca_{1/2,d}$

We now give a sufficient condition for tightness of the partial sum process 
associated to a strictly stationary random field, that is, a random field 
$\pr{X_{\gri}}_{\gri\in\Z^d}$ such that for each integer $N$ and each 
$\gr{i_1},\dots,\gr{i_N},\gr{j}\in\Z^d$, the vectors 
$\pr{X_{\gr{i_k}+\grj}}_{k=1}^N$ and $\pr{X_{\gr{i_k}}}_{k=1}^N$ 
have the same distribution.
The criterion puts into play tails of the maximum of partial 
sums of the rectangles.
 
\begin{Proposition}\label{prop:tightness_partial_sums}
Let $d\geq 1$ and $\rho$ be an element of $\Rca_{1/2,d}$. 
If $\pr{X_{\gri}}_{\gri\in\Z^d}$ is a strictly stationary random field such that 
for each $q\in \ens{1,\dots,d}$ 
 and each positive $\varepsilon$,
\begin{equation}\label{eq:tightness_random_fields_sufficient_cond}
 \lim_{J\to \infty}\limsup_{\min\grm\to \infty}
 \sum_{j=J}^{m_q}2^j \PP\ens{\max_{\gr{1}\imd\gr{k}\imd 2^{\grm -m_q\gr{e}_q } }
  \abs{\sum_{\gr{1}\imd\gri\imd\grk} X_{\gri}   }> \varepsilon\rho\pr{2^{-j}}
  \prod_{u=1}^d2^{m_u/2}}=0,
\end{equation}
then the net $\left(W_{\grn}\right)_{\grn \smd \gr{1}}$
is asymptotically tight in $\mathcal H_{\rho}([0,1]^d)$ as $\min\grn=
\min_{1\leq q\leq d}n_q\to\infty$.
\end{Proposition}

It turns out that it is more convenient to consider maximum indexed 
by dyadic elements of $\Z^d$. In order to avoid confusion, we will 
use notations of the form $2^{m_k}$ instead of $2^{n_k}$.

A similar tightness criterion was used in \cite{MR3615086} 
for sequences, giving optimal results for mixing sequences and allowing to 
recover the optimal result for i.i.d.\ sequences. For random fields, 
the inequality \eqref{eq:exp_deviation_inequality_orthomartingale} 
we obtained and that in Theorem~1.13 of \cite{MR4046858} are appropriate tools 
to check \eqref{eq:tightness_random_fields_sufficient_cond}.

No other assumption than stationarity is done but of course, some dependence 
will be required for this condition to be 
satisfied since one need a good control of the tails of the partial sums on 
rectangles normalized by the square root of the 
number of elements in the rectangle.

Now, we state a result for the weak convergence of 
$\pr{W_{\grn}}_{\grn\smd\gr{1}}$ in the space $H_\rho$, 
where $\rho\in\Rca_{1/2,d}$. 

\begin{Theorem}
\label{thm:WIP_Holder}
Let $d\geq 1$ and let $\pr{X_{\gri}}_{\gri\in\Z^d}$ be a 
strictly stationary orthomartingale difference 
random field. Let $\rho$ be an element of $\Rca_{1/2,d}$ 
be given by \eqref{eq:def_de_rho}, where $L$ is slowly 
varying.

Suppose that 
\begin{equation}\label{eq:cond_suff_WIP}
\forall A>0, \quad\sum_{j\geq 1}2^j\PP\ens{\abs{X_{\gr{1}}}>L\pr{2^j}A}<\infty.
\end{equation}
Then $\pr{W_{\grn}}_{\grn\smd\gr{1}}$ is asymptotically 
tight in $\Hca_\rho\pr{[0,1]^d}$ as $\min\grn=
\min_{1\leq q\leq d}n_q\to\infty$.

Assume moreover that
any $A\in\mathcal B\pr{\R^{\Z^d}}$ such that 
$\ens{\pr{X_{\gri}}_{\gri\in\Z^d}\in 
A}=\ens{\pr{X_{\gri+\gr{e}_d}}_{\gri\in\Z^d}\in A}$ satisfies
$\PP\ens{\pr{X_{\gri}}_{\gri\in\Z^d}\in A}\in\ens{0,1}$.
Then $\pr{W_{\grn}}_{\grn\smd\gr{1}}$ converges 
to $\norm{X_{\gr{0}}}_2 W$ in
$\Hca_\rho\pr{[0,1]^d}$ as $\min\grn=
\min_{1\leq q\leq d}n_q\to\infty$, where $W$ is a standard Brownian sheet.
\end{Theorem}

\section{Proofs}\label{sec:proofs}

\subsection{Proof of Theorem~\ref{thm:deviation_inequality_orthomartingale}}

The proof will be done by induction on the dimension $d$. 
For $d=1$, we need the following deviation inequality for 
 one dimensional martingales. This is a combination of 
Theorem~2.1 in \cite{MR3311214} and 
Theorem~6 in \cite{MR606989}.

\begin{Proposition}[Proposition~2.1 in 
 \cite{MR4466386}]
 \label{prop:inegalite_deviation_martingales_dominee}
 Let $\pr{D_i}_{i\geq 1}$ be a martingale difference sequence 
with respect to the filtration $\pr{\Fca_i}_{i\geq 0}$.  Suppose that $\E{ D_i 
^2}$ is finite for all 
$i\geq 1$. Suppose that there exists a nonnegative random variable $Y$ such 
that 
for all $1\leq i\leq n$, $\E{\varphi\pr{ D_i^2}}\leq\E{\varphi\pr{Y^2}}$ for all 
convex increasing function $\varphi\colon 
[0,\infty)\to[0,\infty)$. 
 Then for all $x,y>0$ and each $n\geq 1$, the following 
inequality holds:
 
\begin{equation}
\label{eq:inegalite_max_martingales_moments_ordre_p_domination_sto}
 \PP\ens{\max_{1\leq k\leq n}\abs{\sum_{i=1}^kD_i}>x n^{1/2}  }
 \leq 2\exp\pr{-\frac 12\pr{\frac{x}{y}}^{2}  }\\ 
 + 4\int_{1 }^{+\infty}w\PP\ens{ Y>yw/2}\mathrm{d}w.
 \end{equation}
\end{Proposition}

Assume that Theorem~\ref{thm:deviation_inequality_orthomartingale} holds in 
dimension $d-1$ with $d\geq 2$. We have to prove 
\eqref{eq:exp_deviation_inequality_orthomartingale} 
for all $d$-dimensional orthomartingale difference 
fields satisfying \eqref{eq:invariance_loi_sommes}.
We first get rid of the maximum over the coordinate $d$, 
apply the inequality of the $d-1$ dimension to the obtained 
orthomartingale and we are then reduced to control the 
tails of a one dimensional martingale. More concretely, 
the induction step is done as follows.
\begin{enumerate}
 \item Step 1: let $M:=\max_{\gr{1}\imd\gri\imd\gr{n}  } \abs{S_{\gri}  }$ and 
 \begin{equation}
  M':=\max_{1\leq i_q\leq n_q,1\leq q\leq d-1}
  \abs{S_{i_1,\dots,i_{d-1},n_d}}.
 \end{equation}
Then we will show that 
\begin{equation}
 \PP\ens{M>x}\leq \int_1^\infty\PP\ens{M'>xu/2}du.
\end{equation}
\item Step 2: the tails of $M'$ are controlled by 
applying the result for $\pr{d-1}$-dimensional random fields. 
\item Step 3: it remains to control the tails of 
$\sum_{i_d=1}^{n_d}X_{1,\dots,1,i_d}$, which can be done 
by using the one dimensional result.
\end{enumerate}

The proof will be quite similar to that of 
Theorem~1.1 in \cite{MR4466386}. The latter 
gave an exponential inequality in the spirit of those of 
the paper for $U$-statistics, that is, sum of terms of the 
form $h\pr{X_1,\dots,X_r}$ where $\pr{X_i}_{i\geq 1}$ is 
i.i.d.\ The connection with orthomartingales is the following. 
First, using decoupling (see \cite{MR1334173}), the 
following inequality takes place:
\begin{equation}
 \PP\ens{\abs{\sum_{1\leq i_1<\dots i_r\leq n}
 h\pr{X_{i_1},\dots,X_{i_r}}
 }>x}\leq C\PP\ens{\abs{\sum_{1\leq i_1<\dots i_r\leq n}
 h\pr{X_{i_1}^{\pr{1}},\dots,X_{i_r}^{\pr{r}}}
 }>Cx},
\end{equation}
where $C$ is a universal constant and the 
sequences $\pr{X_i^{\pr{k}}}_{i\in \Z}, 1\leq k\leq r$ 
are mutually independent copies of 
$\pr{X_i}_{i\in\Z}$. If we assume that 
$\E{h\pr{X_1,\dots,X_r}\mid \sigma\pr{X_i,
i\in\ens{1,\dots,r}\setminus \ens{k}}}=0$ for
all $k\in \ens{1,\dots,r}$, then the random field 
$\pr{ h\pr{X_{i_1}^{\pr{1}},\dots,X_{i_r}^{\pr{r}}}}_{
i_1,\dots,i_r\in\Z}$ is an orthomartingale difference 
random field for the filtration 
$\pr{\Fca_{i_1,\dots,i_r}}_{i_1,\dots,i_r\in\Z}$ given by 
$\Fca_{i_1,\dots,i_r}=\sigma\pr{\eps_{j_q}^{\pr{q}},
j_q\leq i_q}$.

Let us now go into the details of the proof. Let $x$, $y>0$.
We can assume without loss of generality  that $x/y>3^d$. Indeed, suppose 
that we showed the existence of constants $A'_d$, $B_d$ and $C_d$ such that the 
inequality
\begin{multline} 
 \PP\ens{\max_{\gr{1}\imd\gri\imd\gr{n}  } \abs{S_{\gri}  }
 >x\abs{\grn}^{1/2}}\leq A'_d\exp\pr{-\pr{\frac{x}{y}}^{2/d}}
 \\+B_d\int_1^{+\infty}\PP\ens{\abs{X_{\gr{1}}}>yuC_d}
 u\pr{\log\pr{1+u}}^{p_d}du,
\end{multline}
holds for each orthomartingale difference random field 
$\pr{X_{\gri}}_{\gri\in\Z^d}$ satisfying \eqref{eq:invariance_loi_sommes} and 
for each $x,y>0$ such that $x/y>3^d$, then replacing $A'_d$ by 
$A_d=\max\ens{\exp\pr{9},A'_d}$ ensures that $A_d\exp\pr{-\pr{x/y}^{2/d}}
\geq  1$ when $x/y \leq 3^d$, so that the desired inequality becomes trivial in 
this range of parameters.
\begin{enumerate}
 \item Step 1. Let $M$ and $M_N$ be defined as 
  $M:=\max_{\gr{1}\imd\gri\imd\gr{n}  } \abs{S_{\gri}  }$, 
  $\grn=\pr{n_1,\dots,n_d}$ and 
 \begin{equation}
  M_N:=\max_{1\leq i_q\leq n_q,1\leq q\leq d-1}
  \abs{S_{i_1,\dots,i_{d-1},N}}.
 \end{equation}
 Define the filtration $\Gca_N$ as $\Fca_{n_1,\dots,n_{d-1},
 N}$. We check that $\pr{M_N}_{N\geq 1}$ is a submartingale 
 with respect the filtration $\pr{\Gca_N}_{N\geq 1}$. Indeed, 
 \begin{equation}
  \E{M_N\mid \Gca_{N-1}}
  \geq \max_{1\leq i_q\leq n_q,1\leq q\leq d-1}
  \abs{\E{S_{i_1,\dots,i_{d-1},N}}\mid\Gca_{N-1}    }
 \end{equation}
and by the orthomartingale property, 
\begin{equation}
 \E{S_{i_1,\dots,i_{d-1},N}\mid\Gca_{N-1}}
 =S_{i_1,\dots,i_{d-1},N-1}.
\end{equation}
By Doob's inequality, we derive that 
\begin{equation}
x\PP\ens{M>x}\leq \E{M_{n_d}\mathbf{1}\ens{M>x}}. 
\end{equation}
Expressing the latter expectation as an integral of 
the tail and cutting this integral at $x/2$ gives the bound 
\begin{equation}\label{eq:first_step_induction}
 \PP\ens{\max_{\gr{1}\imd\gri\imd\gr{n}  } \abs{S_{\gri}  }
 >x\abs{\grn}^{1/2}} \leq
 \int_1^{+\infty}
 \PP\ens{
 \max_{1\leq i_q\leq n_q,1\leq q\leq d-1}
  \abs{S_{i_1,\dots,i_{d-1},n_d}}>x\abs{\grn}^{1/2}u/2
 }du.
\end{equation}
\item Step 2. Define for $i_1,\dots,i_{d-1}$ the random 
variable 
\begin{equation}
 X'_{i_1,\dots,i_{d-1}}:=
 \sum_{i_d=1}^{n_d}X_{i_1,\dots,i_{d-1},i_d}.
\end{equation}
Then $\pr{X'_{i_1,\dots,i_{d-1}}}_{i_1,\dots,i_{d-1}\in\Z}$ 
is an orthomartingale difference random field with respect to 
the commuting filtration $\pr{\Fca_{i_1,\dots,i_{d-1},n_d}}_{i_1,\dots,i_{d-1}\in\Z}$
satisfying \eqref{eq:invariance_loi_sommes}. We thus apply 
the induction hypothesis with 
\begin{equation}
\widetilde{x}:= xn_d^{1/2}u/2  ,       \quad \widetilde{y}:=  n_d^{1/2}x^{ 1/d}y^{\frac{d-1}d}u
\pr{1+2\ln\pr{ u}}^{-\frac{d-1}2}/2.
\end{equation}
We get in view of \eqref{eq:first_step_induction} and 
\begin{equation}
 \pr{\frac{\widetilde{x}}{\widetilde{y}}}^{\frac{2}{d-1}}= 
 \pr{\frac xy}^{2/d}\pr{1+2\ln u}>3
\end{equation}

that 
\begin{multline}\label{eq:second_step_induction}
 \PP\ens{\max_{\gr{1}\imd\gri\imd\gr{n}  } \abs{S_{\gri}  }
 >x\abs{\grn}^{1/2}}  \leq
 A_{d-1}\int_1^{+\infty}\exp\pr{-\pr{\frac xy}^{2/d}\pr{1+2\ln u}}du\\+ 
 B_{d-1}\int_1^{+\infty}\int_1^{+\infty}\PP\ens{
  \abs{X'_{1,\dots,1}}>\frac{C_{d-1}}2vn_d^{1/2}x^{1/d}y^{\frac{d-1}d}u
\pr{1+2\ln\pr{ u}}^{-\frac{d-1}2}
 }v\pr{\log\pr{1+v}}^{p_{d-1}}dudv.
\end{multline}
For the first term, using that $ \pr{\frac xy}^{2/d}>3$, 
we get 
\begin{equation}
 \int_1^{+\infty}\exp\pr{-\pr{\frac xy}^{2/d}\pr{1+2\ln u}}du
 \leq \exp\pr{-\pr{\frac xy}^{2/d}}
  \int_1^{+\infty}\exp\pr{-6\ln u}du
\end{equation}
hence 
\begin{multline}\label{eq:second_step_induction_final}
 \PP\ens{\max_{\gr{1}\imd\gri\imd\gr{n}  } \abs{S_{\gri}  }
 >x\abs{\grn}^{1/2}}  \leq
 \frac{A_{d-1}}5\exp\pr{-\pr{\frac xy}^{2/d}}\\+ 
 B_{d-1}\int_1^{+\infty}\int_1^{+\infty}\PP\ens{
  \abs{\sum_{i_d=1}^{n_d}X_{1,\dots,1,i_d}}>\frac{C_{d-1}}2vn_d^{1/2}x^{1/d}y^{\frac{d-1}d}u
\pr{1+2\ln\pr{ u}}^{-\frac{d-1}2}
 }v\pr{\log\pr{1+v}}^{p_{d-1}}dudv.
\end{multline}
\item Step 3. It remains to find a bound for the double 
integral. For fixed $u$ and $v>1$, we apply 
Proposition~\ref{prop:inegalite_deviation_martingales_dominee} 
in the following setting: $x$ is replaced by 
$\bar{x}$ defined by 
\begin{equation}
 \bar{x}:=\frac{C_{d-1}}2v x^{1/d}y^{\frac{d-1}d}u
\pr{1+2\ln\pr{ u}}^{-\frac{d-1}2}
\end{equation}
and $y$ by $\bar{y}$ defined by 
\begin{equation}
 \bar{y}:= \frac{C_{d-1}}2yuv\pr{1+2\ln u}^{-\frac{d-1}2}
 2^{-1/2}\pr{1+2\ln\pr{uv}}^{-1/2}.
\end{equation}
In this way, 
\begin{equation}
 \frac 12\pr{\frac{\bar{x}}{\bar{y}}}^2
 =\pr{\frac{x}y}^{2/d}\pr{1+2\ln\pr{uv}}
\end{equation}
and \eqref{eq:inegalite_max_martingales_moments_ordre_p_domination_sto} gives 
\begin{multline}\label{eq:estimated_last_step}
 \PP\ens{\max_{\gr{1}\imd\gri\imd\gr{n}  } \abs{S_{\gri}  }
 >x\abs{\grn}^{1/2}}  \leq
 \frac{A_{d-1}}5\exp\pr{-\pr{\frac xy}^{2/d}}\\+ 
 2B_{d-1}\exp\pr{-\pr{\frac{x}{y}}^{2/d}}\int_1^{+\infty}\int_1^{+\infty}
 \exp\pr{-2\ln\pr{uv}}dudv\\
 + 4B_{d-1}\int_1^{+\infty}\int_1^{+\infty}\int_1^{+\infty}
 h\pr{u,v,w}dudvdw,
\end{multline}
where 
\begin{equation}
 h\pr{u,v,w}:= 
 \PP\ens{\abs{X_{\gr{1}}}> 
 \frac{C_{d-1}}4yuvw\pr{1+2\ln u}^{-\frac{d-1}2}
 2^{-1/2}\pr{1+2\ln\pr{uv}}^{-1/2}
 }  vw\pr{\log\pr{1+v}}^{p_{d-1}}.
\end{equation}
Observing that for $u,v>1$, $\pr{1+2\ln\pr{uv}}^{-1/2}\geq 
\pr{1+2\ln\pr{u}}^{-1/2}\pr{1+2\ln\pr{v}}^{-1/2}$ and defining 
the functions
\begin{equation}
f_q\pr{t}=t\pr{1+2\ln t}^{-q}, t\geq 1, q>0
\end{equation}
\begin{equation}
g\colon t\mapsto  \PP\ens{\abs{X_{\gr{1}}}> 
 \frac{C_{d-1}}{4\cdot 2^{1/2}}yt
 }  ,
\end{equation}
we derive from estimate \eqref{eq:estimated_last_step} that 
\begin{multline}\label{eq:estimated_last_step_bis}
 \PP\ens{\max_{\gr{1}\imd\gri\imd\gr{n}  } \abs{S_{\gri}  }
 >x\abs{\grn}^{1/2}}  \leq
 \frac{A_{d-1}}5\exp\pr{-\pr{\frac xy}^{2/d}}+ 
 2B_{d-1}\exp\pr{-\pr{\frac{x}{y}}^{2/d}} \\
 + 4B_{d-1}\int_1^{+\infty}\int_1^{+\infty}\int_1^{+\infty}
g\pr{w f_{d/2}\pr{u}f_{1/2}\pr{v}}  vw\pr{\log\pr{1+v}}^{p_{d-1}}dwdudv.
\end{multline}
Doing for fixed $u,v\geq 1$ the substitution $t=w f_{d/2}\pr{u}f_{1/2}\pr{v}$, 
we are reduced to show that there exists a constant $K$ such that for each 
$t\geq 1$,
\begin{equation}
I\pr{t}\leq K\pr{ \log\pr{1+t} }^{p_d},
\end{equation}
where 
\begin{equation}
I\pr{t}:= \int_1^{+\infty}\int_1^{+\infty}\frac{v\pr{\log\pr{1+v}}^{p_{d-1}} }{\pr{f_{d/2}\pr{u}f_{1/2}\pr{v}}^2}
\mathbf{1}_{t>f_{d/2}\pr{u}f_{1/2}\pr{v}}du dv.
\end{equation}

Using the fact that there exists a constant $c$ such that for each $u\geq 1$,  $f_{d/2}\pr{u}\geq c$, we derive that 
\begin{equation}
I\pr{t}\leq  \int_1^{+\infty}\int_1^{+\infty}\frac{v\pr{\log\pr{1+v}}^{p_{d-1}} }{\pr{f_{d/2}\pr{u}f_{1/2}\pr{v}}^2}
\mathbf{1}_{t>cf_{1/2}\pr{v} }du dv
\end{equation}
and since the integral over $u$ is convergent, we are reduced to show that 
 there exists a constant $K$ such that for each $t\geq 1$,
\begin{equation} 
\int_1^{+\infty}\frac1v\pr{\log\pr{1+2v}}^{p_{d-1}+1}\mathbf{1}_{t>cf_{1/2}
\pr{v
} }dv  \leq K\pr{ \log\pr{1+t} }^{p_d}.
\end{equation}
Since there exists a constant $\kappa$ such that for each $v\geq 1$, $f_{1/2}\pr{v}\geq \kappa\sqrt v$, one can see that 
\begin{equation} 
 \int_1^{+\infty}\frac1v\pr{\log\pr{1+2v}}^{p_{d-1}+1}\mathbf{1}_{t>cf_{1/2}\pr{v} }dv  \leq \mathrm{1}_{t>c\kappa}\int_1^{\pr{t/(c\kappa)}^2}
 \frac1v\pr{\log\pr{1+2v}}^{p_{d-1}+1}dv
\end{equation}
and the fact that $p_d=p_{d-1}+2$ ends the proof of 
Theorem~\ref{thm:deviation_inequality_orthomartingale}.
\end{enumerate}

\subsection{Proof of Theorem~\ref{thm:large_deviation}}

We apply Theorem~\ref{thm:deviation_inequality_orthomartingale} 
to $y=\abs{\gr{N}}^{1/\pr{2+d\gamma}}x^{2/\pr{2+d\gamma}}$. 
Bounding the resulting integral term by a constant times 
$\exp\pr{-y^\gamma}$, one sees that in this case, the 
two terms of the right hand side of \eqref{eq:exp_deviation_inequality_orthomartingale} have 
a similar contribution.

\subsection{Proof of Proposition~\ref{prop:tightness_partial_sums}}

The proof of the tightness criterion rests on the Schauder decomposition of the 
spaces 
$\mathcal H_\rho^o\left([0,1]^d\right)$. In order to state it, we need to introduce the following notations.

Set for $j\geqslant 0$, 
\begin{equation}
 W_j:=\ens{k2^{-j},0\leq k\leq 2^j}^d
\end{equation}
and 
\begin{equation}
 V_0:=W_0, \quad V_j:=W_j\setminus W_{j-1},j\geqslant 1.
\end{equation}

We define for $\gr{v}\in V_j$ the pyramidal function $\Lambda_{j,
\gr{v}}$ by 
\begin{equation}
 \Lambda_{j,\gr{v}}(\gr{t}):=\Lambda(2^j(\gr{t}-\gr{v})), \quad 
 \gr{t}\in [0,1]^d,
\end{equation}
where 
\begin{equation}
 \Lambda(\gr{t}):=\max\ens{0,1-\max_{t_i<0}\abs{t_i}-\max_{t_i>0}\abs{t_i}}, 
 \quad \gr{t}=(t_i)_{i=1}^d\in [-1,1]^d.
\end{equation}
For $x\in \mathcal H_\alpha^o([0,1]^d)$, we define the coefficients 
$\lambda_{j,\gr{v}}(x)$ by 
$\lambda_{0,\gr{v}}(x)=x(\gr{v})$, $\gr{v}\in V_0$ and for 
$j\geqslant 1$ and $v\in V_j$, 
\begin{equation}
 \lambda_{j,\gr{v}}(x):=x\left(\gr{v}\right)-\frac 12\left(x\left(\gr{v^-}\right)+
 x\left(\gr{v^+}\right)\right),
\end{equation}
where $\gr{v^+}$ and $\gr{v^-}$ are define in the following way. 
Each $\gr{v}\in V_j$ is represented in a unique way by $\gr{v}=
\left(k_i2^{-j}\right)_{i=1}^d$.
Then $\gr{v^+}:=(v^+_i)_{i=1}^d$ and $\gr{v^-}:=\left(v^-_i\right)_{i=1}^d$ are defined by 
\begin{equation}
 \gr{v_i^-}:=\begin{cases}
         v_i-2^{-j},&\mbox{ if }k_i\mbox{ is odd;}\\
         v_i,&\mbox{ if }k_i\mbox{ is even}
        \end{cases}\quad 
 \gr{v_i^+}:=\begin{cases}
         v_i+2^{-j},&\mbox{ if }k_i\mbox{ is odd;}\\
         v_i,&\mbox{ if }k_i\mbox{ is even.}
        \end{cases}
\end{equation}
The sequential norm is defined by 
\begin{equation}
 \norm{x}_\rho^{\mathrm{seq}}:=
 \sup_{j\geqslant 0}\rho\pr{2^{-j}}^{-1}\max_{\gr{v}\in V_j}
 \abs{\lambda_{j,\gr{v}}(x)}, \quad x\in \mathcal H_\rho^o\left([0,1]^d\right).
\end{equation}
By \cite{MR2141332}, the norm $\norm{\cdot}_\rho^{\mathrm{seq}}$ is equivalent 
to $\norm{\cdot}_\rho$ on $\mathcal H_\rho^o\left([0,1]^d\right)$.

 A general tightness criterion is available for moduli of 
 the form $\rho\colon h\mapsto h^\alpha$. The criterion rests on a Schauder 
 decomposition of $\mathcal H_\rho^o$ as $\bigoplus_{j\geq 1}E_j$, 
 where $E_j$ is the vector space generated by the functions 
 $\Lambda_{j,\gr{v}}$, $\gr{v}\in V_j$.

\begin{Theorem}[Theorem~6, \cite{MR2340883}]\label{thm:general_theorem_tightness_Hölder_dim_d}
 Let $\ens{\zeta_{\grn},\grn\in \N^d}$ and $\zeta$ be random elements with values in the 
 space $\mathcal H_\alpha\left([0,1]^d\right)$. Assume that the following conditions are 
 satisfied. 
 \begin{enumerate}
  \item For each dyadic $\gr{t}\in [0,1]^d$, the net $\ens{\zeta_{\grn}(\gr{t}),
  \grn\in \N^d}$ 
  is asymptotically tight on $\R$. 
  \item For each positive $\varepsilon$, 
  \begin{equation}
   \lim_{J\to \infty}\limsup_{\min\grn\to \infty}
   \PP\ens{\sup_{j\geqslant J}2^{\alpha j}\max_{\gr{v}\in V_j}\abs{\lambda_{j,
   \gr{v}}(\zeta_{\grn})
   }>\varepsilon}=0.
  \end{equation}
 \end{enumerate}
Then the net $\ens{\zeta_{\grn},\grn\in \N^d}$ is asymptotically tight in the space 
$\mathcal H_\alpha([0,1]^d)$.
\end{Theorem}
This extends readily to $\rho\in\mathcal R_{1/2,d}$.

First observe that by bounding from below the sum over $j$ 
by the term at index $j=J$ and taking $i=d$, 
\eqref{eq:tightness_random_fields_sufficient_cond} 
implies that 
\begin{equation} 
 \lim_{J\to \infty}\limsup_{\min\grm\to \infty}
  2^J \PP\ens{\max_{\gr{1}\imd\gr{k}\imd 2^{\grm -J\gr{e}_d } }
  \abs{\sum_{\gr{1}\imd\gri\imd\grk} X_{\gri}   }> \varepsilon\rho\pr{2^{-j}}
  \prod_{u=1}^d2^{m_u/2}}=0,
\end{equation}
and doing the replacement of index $m'_d=m_d-J$ for a 
fixed $J$ gives asymptotic tightness of 
$ \pr{W_{\grn}\pr{\gr{t}}}_{\grn\smd\gr{1}}$ for each 
$\gr{t}$.

It remains to check the second condition of Theorem~\ref{thm:general_theorem_tightness_Hölder_dim_d}.
Since the Schauder decomposition $\mathcal H_\rho^o=
\bigoplus_{j\geq 1}E_j$ is also valid for $\rho\in\Rca_{1/2,d}$, 
this theorem also holds with the map $h\mapsto h^\alpha$ replaced 
by $\rho$. Therefore, 
it suffices to prove that 
 \begin{equation}\label{eq:step_for_tightness}
   \lim_{J\to \infty}\limsup_{\min\grn\to \infty}
   \PP\ens{\sup_{j\geqslant J} \rho\pr{2^{-j}}^{-1}\max_{\gr{v}
   \in V_j}\abs{\lambda_{j,
   \gr{v}}\pr{W_{\grn}}
   }>\varepsilon }=0.
  \end{equation}

Consider for $\gr{s}:=(s_2,\dots,s_d)\in [0,1]^{d-1}$ and $t,t'\in [0,1]$ the 
quantity 
\begin{equation}
 \Delta_{\grn}(t,t',\gr{s}):=\sqrt{\abs{\grn}}
 \abs{W_{\grn}\pr{t',\gr{s}}-W_{\grn}\pr{t,\gr{s}}}.
\end{equation}

We recall the following lemma:
\begin{Lemma}[Lemma~11, \cite{MR2340883}]\label{lem:Rac_Suq_Zem}
 For any $t',t\in [0,1]$, $t'>t$, the following inequality holds:  
 \begin{multline}
  \sup_{\gr{s}\in [0,1]^{d-1}}\Delta_{\grn}\pr{t,t',\gr{s}}
  \leq  3^d\mathbf 1\ens{t'-t\geqslant \frac 1{n_1}}\quad 
  \max_{\mathclap{\substack{ 1\leq k_\ell\leq n_\ell \\
  2\leq \ell\leq d} }}\quad\abs{\sum_{i_1=[n_1t]+1}^{[n_1t']}\quad
  \sum_{\mathclap{\substack{ 1\leq i_\ell\leq k_\ell \\
  2\leq \ell\leq d} }}\quad X_{\gri}}+\\
  +3^d\min\ens{1,n_1(t'-t)}\max_{1\leq i_1\leq n_1}
  \quad 
  \max_{\mathclap{\substack{ 1\leq k_\ell\leq n_\ell \\
  2\leq \ell\leq d} }}\quad\abs{\quad
  \sum_{\mathclap{\substack{ 1\leq i_\ell\leq k_\ell \\
  2\leq \ell\leq d} }}\quad X_{\gri}}.
 \end{multline}
\end{Lemma}

Now, we define for $q\in [d]$ and 
$\gr{s}=\pr{s_\ell}_{\ell\in [d] \setminus \ens q}\in [0,1]^{d-1}$, 
\begin{equation*}
 \Delta_{\grn}^{(q)}(t,t',\gr{s}):=
 \abs{W_{\grn}\pr{s_1,\dots,s_{q-1},t',s_{q+1},\dots,s_d}-
 W_{\grn}\pr{s_1,\dots,s_{q-1},t,s_{q+1},\dots,s_d}}.
\end{equation*}
 By definition of $\lambda_{j,\gr{v}}$ and $V_j$, the inequality 
\begin{multline}\label{eq:Hölder_modulus_random_fields_probability_bound_proof}
  \PP\ens{\sup_{j\geqslant J}\rho\pr{2^{-j}}^{-1}\max_{\gr{v}\in V_j}\abs{\lambda_{j,\gr{v}}\pr{W_{\grn}}
   }>2^{d+1} \eps} \\ \leq
 \sum_{q=1}^d\PP\ens{\sup_{j\geqslant J}\rho\pr{2^{-j}}^{-1}
 \max_{\mathclap{\substack{0\leq k<2^j\\
 \gr{0} \imd \gru\imd 2^j\gr{1}}}} \Delta^{(q)}_{\grn}\left(t_{k+1},t_k;
 \gr{s_u}\right)>2\cdot  \eps}
\end{multline}
takes place, where $t_k=k2^{-j}$ and $\gr{s_u}=\left(u_i2^{-j}\right)_{i\in [d] \setminus 
\ens q}$. We show that for each positive $\eps$, 
\begin{equation}
 \lim_{J\to\infty}\limsup_{\grn\to \infty}
 \PP\ens{\sup_{j\geqslant J}\rho\pr{2^{-j}}^{-1}
 \max_{\mathclap{\substack{0\leq k<2^j\\
 \gr{0} \imd \gru\imd 2^j\gr{1}}}} \Delta^{(d)}_{\grn}\left(t_{k+1},t_k;
 \gr{s_u}\right)>2\cdot 3^d\cdot  \eps},
\end{equation}
the treatment of the corresponding terms with $\Delta^{(q)}$ instead of 
$\Delta^{(d)}$ can be 
done 
by switching the roles of the coordinates. We will use the following notations:
we will denote by $\gri,\grk,\grn,\grm$ elements of $\Z^d$ and 
$\gr{i'},\gr{k'},\gr{n'},\gr{m'}$ elements of $\Z^{d-1}$ and 
$\pr{\pr{\gr{i'},i_d}}$ will denote the element of $\Z^d$ whose first $d-1$ 
coordinates are those of $\gr{i'}$ and $d$th one is $i_d$, and similarly for
other letters. We will denote by $\imd$ the coordinatewise order on $\Z^d$ and 
$\Z^{d-1}$, since there will be no ambiguity. Finally, we will let $\gr{i'}$ 
denote the element of $\Z^{d-1}$ whose coordinates are all equal to $1$.
 
We have in view of Lemma~\ref{lem:Rac_Suq_Zem} that  
\begin{multline}
 \PP\ens{\sup_{j\geqslant J}\rho\pr{2^{-j}}^{-1}
 \max_{\mathclap{\substack{0\leq k<2^j\\
 \gr{0} \imd \gru\imd 2^j\gr{1}}}}\Delta^{\pr{d}}_{\grn}\left(t_{k+1},t_k;
 \gr{s_u}\right)>2\cdot 3^d\eps\abs{\grn}^{1/2}}\\
 \leq \PP\left\{\sup_{j\geqslant J}
 \max_{0\leq a<2^j}\rho\pr{2^{-j}}^{-1} \mathbf 1\ens{2^{-j}\geqslant \frac 
1{n_d}} 
  \max_{\gr{1'}\imd \gr{k'}\imd\gr{n'} }  
 \abs{\sum_{i_d=[n_da2^{-j}]+1}^{[n_d(a+1)2^{-j}]} 
  \sum_{\gr{1'}\imd\gr{i'}\imd\gr{k'} }  X_{\pr{\gr{i'},i_d} } }+\right.\\
  \left.+\min\ens{1,n_d2^{-j}}\max_{1\leq i_d\leq n_d}
 \max_{\gr{1'}\imd \gr{k'}\imd\gr{n'} } \abs{ 
   \sum_{\gr{1'}\imd\gr{i'}\imd\gr{k'} }  X_{\pr{\gr{i'},i_d} }}
 >2\eps\abs{\grn}^{1/2}\right\} \\
 \leq \PP\left\{\sup_{j\geqslant J}\max_{0\leq a<2^j}\rho\pr{2^{-j}}^{-1}
 \mathbf 1\ens{2^{-j}\geqslant \frac 1{n_d}} 
  \max_{\gr{1'}\imd \gr{k'}\imd\gr{n'} }  
 \abs{\sum_{i_d=[n_da2^{-j}]+1}^{[n_d(a+1)2^{-j}]} 
  \sum_{\gr{1'}\imd\gr{i'}\imd\gr{k'} }  X_{\pr{\gr{i'},i_d} } 
}>\eps\abs{\grn}^{1/2}\right\}+\\
  +\PP\left\{\sup_{j\geqslant J}\rho\pr{2^{-j}}^{-1}
  \min\ens{1,n_d2^{-j}}\max_{1\leq i_d\leq n_d}
 \max_{\gr{1'}\imd \gr{k'}\imd\gr{n'} } \abs{ 
   \sum_{\gr{1'}\imd\gr{i'}\imd\gr{k'} }  X_{\pr{\gr{i'},i_d} 
}}>\eps\abs{\grn}^{1/2}\right\}.
  \label{eq:Hölder_modulus_random_fields_probability_bound_proof_2}
\end{multline}
Since the indicator in the first term of the right hand side of 
\eqref{eq:Hölder_modulus_random_fields_probability_bound_proof_2} vanishes 
if $j>\log n_d$, we have 
\begin{multline}
 \PP\left\{\sup_{j\geqslant J}\max_{0\leq a<2^j}\rho\pr{2^{-j}}^{-1} 
 \mathbf 1\ens{2^{-j}\geqslant \frac 1{n_d}} \max_{\gr{1'}\imd 
\gr{k'}\imd\gr{n'} }  
 \abs{\sum_{i_d=[n_da2^{-j}]+1}^{[n_d(a+1)2^{-j}]} 
  \sum_{\gr{1'}\imd\gr{i'}\imd\gr{k'} }  X_{\pr{\gr{i'},i_d} } 
}>\eps\abs{\grn}^{1/2}\right\}\\
 \leq \PP\left\{\sup_{J\leq j\leq  \log n_d}\max_{0\leq 
a<2^j}\rho\pr{2^{-j}}^{-1}
 \max_{\gr{1'}\imd \gr{k'}\imd\gr{n'} }  
 \abs{\sum_{i_d=[n_da2^{-j}]+1}^{[n_d(a+1)2^{-j}]} 
  \sum_{\gr{1'}\imd\gr{i'}\imd\gr{k'} }  X_{\pr{\gr{i'},i_d} } 
}>\eps\abs{\grn}^{1/2}\right\}\\
  \leq \sum_{j=J}^{\log n_d}2^j
  \max_{0\leq a <2^j}\PP\ens{
  \rho\pr{2^{-j}}^{-1} 
 \max_{\gr{1'}\imd \gr{k'}\imd\gr{n'} }  
 \abs{\sum_{i_d=[n_da2^{-j}]+1}^{[n_d(a+1)2^{-j}]} 
  \sum_{\gr{1'}\imd\gr{i'}\imd\gr{k'} }  X_{\pr{\gr{i'},i_d} } 
}>\eps\abs{\grn}^{1/2}
  },
\end{multline}
and by stationarity, it follows that 
\begin{multline}
 \PP\left\{\sup_{j\geqslant J}\max_{0\leq a<2^j}\rho\pr{2^{-j}}^{-1} 
 \mathbf 1\ens{2^{-j}\geqslant \frac 1{n_d}} \max_{\gr{1'}\imd 
\gr{k'}\imd\gr{n'} }  
 \abs{\sum_{i_d=[n_da2^{-j}]+1}^{[n_d(a+1)2^{-j}]} 
  \sum_{\gr{1'}\imd\gr{i'}\imd\gr{k'} }  X_{\pr{\gr{i'},i_d} } 
}>\eps\abs{\grn}^{1/2}\right\}\\
 \leq \sum_{j=J}^{\log n_d}2^j
  \PP\ens{\rho\pr{2^{-j}}^{-1} 
 \max_{1\leq k_d\leq 2n_d2^{-j} } \max_{\gr{1'}\imd \gr{k'}\imd\gr{n'} }  
 \abs{\sum_{i_d= 1}^{k_d} 
  \sum_{\gr{1'}\imd\gr{i'}\imd\gr{k'} }  X_{\pr{\gr{i'},i_d} } 
}>\eps\abs{\grn}^{1/2}
  }.
\end{multline}
If $\grn =\pr{n_1,\dots,n_d}$ is such that $2^{m_i}\leq n_i\leq  
2^{m_i+1}-1$ for each $i\in [d]$, then we derive that 
\begin{multline}
\label{eq:Hölder_modulus_random_fields_probability_bound_proof_3}
 \PP\left\{\sup_{j\geqslant J}\max_{0\leq a<2^j}\rho\pr{2^{-j}}^{-1} 
 \mathbf 1\ens{2^{-j}\geqslant \frac 1{n_d}} \max_{\gr{1'}\imd 
\gr{k'}\imd\gr{n'} }  
 \abs{\sum_{i_d=[n_da2^{-j}]+1}^{[n_d(a+1)2^{-j}]} 
  \sum_{\gr{1'}\imd\gr{i'}\imd\gr{k'} }  X_{\pr{\gr{i'},i_d} } 
}>\eps\abs{\grn}^{1/2}\right\}\\
 \leq \sum_{j=J}^{m_d+1}2^j
  \PP\ens{\rho\pr{2^{-j}}^{-1} 
  \max_{\gr{1}\imd \gr{k}\imd\gr{2^{\grm-
 \pr{j-2}\gr{e_d}}} }  
 \abs{ 
  \sum_{\gr{1}\imd\gr{i}\imd\gr{k} }  X_{\pr{\gr{i'},i_d} } 
}>\eps \prod_{u=1}^d 2^{m_u/2}
  }.
\end{multline}
For the second term of the right hand side of 
\eqref{eq:Hölder_modulus_random_fields_probability_bound_proof_2}, notice that 
\begin{equation}
 \sup_{j\geqslant J}\rho\pr{2^{-j}}^{-1}\min\ens{1,n_d2^{-j}}\leq c_\rho 
\rho\pr{1/n_d}^{-1}.
\end{equation}
Indeed, if $j\leq \log n_d$, then $2^j\leq n_d$ hence 
$\rho\pr{2^{-j}}^{-1}\min\ens{1,n_d2^{-j}}=\rho\pr{2^{-j}}^{-1}$ and the 
sequence $\pr{\rho\pr{2^{-j}}^{-1}}_j$ is increasing, and if $j>\log n_d$, then 
$2^j>n_d$, hence $\min\ens{1,n_d2^{-j}}=n_d2^{-j}$ and for such $j$'s, we have 
$\rho\pr{2^{-j}}^{-1}n_d2^{-j}$ and we use decreasingness of 
the sequence $\pr{\rho\pr{2^{-j}}^{-1}2^{-j}}_j$.
As a consequence, after having bounded the probability of the max over $i_d$ by 
the 
sum of probabilities and used stationarity, we obtain 
\begin{multline}
 \PP\left\{\sup_{j\geqslant J}\rho\pr{2^{-j}}^{-1}
  \min\ens{1,n_d2^{-j}}\max_{1\leq i_d\leq n_d}
 \max_{\gr{1'}\imd \gr{k'}\imd\gr{n'} } \abs{ 
   \sum_{\gr{1'}\imd\gr{i'}\imd\gr{k'} }  X_{\pr{\gr{i'},i_d} 
}}>\eps\abs{\grn}^{1/2}\right\}\\
  \leq n_d
  \PP\left\{ \max_{\gr{1'}\imd \gr{k'}\imd\gr{n'} } \abs{ 
   \sum_{\gr{1'}\imd\gr{i'}\imd\gr{k'} }  
X_{\pr{\gr{i'},1}}}>\eps\abs{\grn}^{1/2}\rho\pr{1/n_d} /c_{\rho}\right\}.
\end{multline}
Notice that if $\grn =\pr{n_1,\dots,n_d}$ is such that $2^{m_i}\leq n_i\leq  
2^{m_i+1}-1$ for each $i\in [d]$, then for 
each $j\leq m_d$, 
\begin{multline}
\frac{1}{\abs{\grn}^{1/2}\rho\pr{1/n_d} }\max_{\gr{1'}\imd \gr{k'}\imd\gr{n'} } 
\abs{ 
   \sum_{\gr{1'}\imd\gr{i'}\imd\gr{k'} }  
X_{\pr{\gr{i'},1}}}  
  \leq 
  \frac{1}{\rho\pr{2^{-m_{d}-1}}\prod_{u=1}^d2^{m_u/2}} 
 \max_{\gr{1'}\imd\gr{k'}\imd 
\gr{2^{m'}}}\abs{\sum_{\gr{1'}\imd\gr{i'}\imd\gr{k'} }  
X_{\pr{\gr{i'},1}}}\\ 
  \leq 
  \frac{1}{\rho\pr{2^{-m_{d}-1}}\prod_{u=1}^d2^{m_u/2}}  
   \max_{1\leq k_d\leq 2^{m_d+1-j}} \max_{\gr{1'}\imd\gr{k'}\imd 
\gr{2^{m'}}}\abs{\sum_{i_d=1}^{k_d}\sum_{\gr{1'}\imd\gr{i'}\imd\gr{k'} }  
X_{\pr{\gr{i'},i_d}}}
\end{multline}
hence 
\begin{multline}
\label{eq:Hölder_modulus_random_fields_probability_bound_proof_4}
 \PP\left\{\sup_{j\geqslant J}\rho\pr{2^{-j}}^{-1}
  \min\ens{1,n_d2^{-j}}\max_{1\leq i_d\leq n_d}
 \max_{\gr{1'}\imd \gr{k'}\imd\gr{n'} } \abs{ 
   \sum_{\gr{1'}\imd\gr{i'}\imd\gr{k'} }  X_{\pr{\gr{i'},i_d} 
}}>\eps\abs{\grn}^{1/2}\right\}\\
  \leq  \sum_{j=J}^{m_d+1}2^j \PP\ens{
 \max_{\gr{1}\imd\gr{k}\imd 
\gr{2^{m-\pr{j-1}\gr{e_d}}}}\abs{\sum_{\gr{1}\imd\gr{i}\imd
\grk }  
X_{\gri}}> \eps\rho\pr{2^{-j}}
  \prod_{u=1}^d2^{m_u/2}/c_\rho}.
\end{multline}

Combining inequalities \eqref{eq:Hölder_modulus_random_fields_probability_bound_proof_2}
with \eqref{eq:Hölder_modulus_random_fields_probability_bound_proof_3} and 
\eqref{eq:Hölder_modulus_random_fields_probability_bound_proof_4}, 
we obtain \eqref{eq:step_for_tightness}
. This ends the proof of 
Proposition~\ref{prop:tightness_partial_sums}.

\subsection{Proof of Theorem~\ref{thm:WIP_Holder}}

We have to check that \eqref{eq:tightness_random_fields_sufficient_cond} is satisfied. 
For simplicity, we will do this for $q=d$; the general case 
can be done similarly. 
To this aim, we apply inequality \eqref{eq:exp_deviation_inequality_orthomartingale} for fixed $\gr{m}\smd\gr{1}$, $J\geq 1$ 
and $j\in\ens{J,\dots,m_d}$ in the following setting: 
$\widetilde{\gr{m}}=\pr{2^{m_1},\dots,2^{m_{d-1}},2^{m_d-j}}$, 
$x=\eps \rho\pr{2^{-j}}2^{j/2}$ and 
$y= \eps/2\rho\pr{2^{-j}}2^{j/2} j^{-d/2}$. The sum of the 
exponential terms in \eqref{eq:exp_deviation_inequality_orthomartingale} can 
be bounded by the remainder of a convergent series. 
With the assumption on $\rho$, the sum of the obtained integral terms does not exceed $\sum_{j\geq J}2^j\int_1^{+\infty}\PP\ens{\abs{X_{\gr{1}}}>L\pr{2^j}uC}
 u\pr{\log\pr{1+u}}^{p_d}du$ where $C$ depends only on $\rho$ and $\eps$.
 
 Now, we will show that \eqref{eq:cond_suff_WIP} 
 guarantees the convergence to zero of the previous term 
 as $J$ goes to infinity. As there exists a constant $\kappa$  such that  $u\pr{\log\pr{1+u}}^{p_d}\leq \kappa u^2$ for each $u\geq 1$, 
 it suffices to prove that for each $C$, 
 \begin{equation}
  \sum_{j\geq 1}2^j\int_1^{+\infty}\PP\ens{\abs{X_{\gr{1}}}>L\pr{2^j}uC}
 u^{2}du<\infty 
 \end{equation}
This will be a consequence of the following lemmas. 
\begin{Lemma}\label{lem:slowly_varying_2j_sur_L2j}
 Let $L\colon \R_{+}\to \R_{+}$ be a slowly varying function. There exists a constant $C_L$ such that for each $k\geq 1$, 
 \begin{equation}\label{eq:slowly_varying_2j_sur_L2j}
  \sum_{j=1}^k\frac{2^j}{L\pr{2^j}}\leq C_L\frac{2^k}{L\pr{2^k}}.
 \end{equation}

\end{Lemma}

\begin{proof}
From Potter's bound (see Lemma 1.5.6. in \cite{MR1015093}), there exists a constant $K$ such that for each $1\leq j\leq k$, $L\pr{2^k} \leq K L\pr{2^j} 2^{\frac{k-j}2}$. Consequently, 
\begin{equation}
\sum_{j=1}^k\frac{2^j}{L\pr{2^j}}
=\sum_{j=1}^k\frac{2^j}{L\pr{2^k}}\frac{L\pr{2^k}}{L\pr{2^j}}\leq K\sum_{j=1}^k\frac{2^j}{L\pr{2^k}}2^{\frac{k-j}2}=K\frac{2^{k/2}}{L\pr{2^k}}\sum_{j=1}^k2^{\frac j2}
\end{equation}
and the change of index $\ell=k-j$ shows that we can take $C_L=K\sum_{\ell\geq 0}2^{-\ell/2}$. This ends the proof of Lemma~\ref{lem:slowly_varying_2j_sur_L2j}.

\end{proof}

\begin{Lemma}\label{lem:slowly_varying}
 Let $X$ be a non-negative random variable and let $L\colon \R_{+}\to \R_{+}$ be a slowly varying increasing function such that $L\pr{x}\to \infty$ 
 as $x\to \infty$. Suppose that 
 \begin{equation}
  \forall A>0, \sum_{j\geq 1}2^j\PP\ens{X>L\pr{2^j}A}<\infty.
 \end{equation}
Then for each $C>0$, 
\begin{equation}
 \sum_{j\geq 1}2^j\int_1^{+\infty}\PP\ens{X>L\pr{2^j}uC}
 u^{2}du<\infty.
\end{equation}
\end{Lemma}
\begin{proof}
 Without loss of generality, we can assume that $C=1$. First, observe that 
 \begin{equation}
  \int_1^{+\infty}\PP\ens{X>L\pr{2^j}u}
 u^{2}du\leq \E{\pr{\frac{X}{L\pr{2^j}}}^{3}\mathbf{1}\ens{X>L\pr{2^j} }}.
 \end{equation}
Therefore, 
\begin{multline}
 \sum_{j\geq 1}2^j\int_1^{+\infty}\PP\ens{X>L\pr{2^j}u}
 u^{2}du \\
 \leq \sum_{j\geq 1}2^j\E{\pr{\frac{X}{L\pr{2^j}}}^{3}\sum_{k\geq j}\mathbf{1}\ens{ L\pr{2^k}<X\leq L\pr{2^{k+1}} }}\\
 \leq \sum_{j\geq 1}2^j\E{\sum_{k\geq j}\pr{\frac{L\pr{2^{k+1}}    }{L\pr{2^j}}}^{3}\mathbf{1}\ens{ L\pr{2^k}<X\leq L\pr{2^{k+1}} }}.
\end{multline}
Switching the sums over $j$ and $k$ and using Lemma~\ref{lem:slowly_varying}  with $L^3$ instead of $L$ reduces use to show that 
\begin{equation}
 \sum_{k\geq 1} 2^k\PP\ens{ L\pr{2^k}<X\leq L\pr{2^{k+1}} }<\infty,
\end{equation}
which follows from the assumption and the obvious bound $\PP\ens{ L\pr{2^k}<X\leq L\pr{2^{k+1}} }\leq 
\PP\ens{ L\pr{2^k}<X }$.
 
\end{proof}

 The last part of the statement of Theorem~\ref{thm:WIP_Holder} 
 follow from \cite{MR3913270}. 
 This ends the proof of Theorem~\ref{thm:WIP_Holder}.

\textbf{Acknowledgement.} The author is grateful to the referee for 
many useful comments and corrections that improved the paper. 

\def\polhk\#1{\setbox0=\hbox{\#1}{{\o}oalign{\hidewidth
  \lower1.5ex\hbox{`}\hidewidth\crcr\unhbox0}}}\def\cprime{$'$}
  \def\polhk#1{\setbox0=\hbox{#1}{\ooalign{\hidewidth
  \lower1.5ex\hbox{`}\hidewidth\crcr\unhbox0}}} \def\cprime{$'$}
\providecommand{\bysame}{\leavevmode\hbox to3em{\hrulefill}\thinspace}
\providecommand{\MR}{\relax\ifhmode\unskip\space\fi MR }
\providecommand{\MRhref}[2]{%
  \href{http://www.ams.org/mathscinet-getitem?mr=#1}{#2}
}
\providecommand{\href}[2]{#2}

\end{document}